\documentclass[a4paper,10pt]{amsart}
\usepackage[utf8]{inputenc}
\usepackage{color}
\usepackage{graphicx}

\newcommand{\highlight}[1]{{#1}}

\newtheorem{theorem}{Theorem}[section]

\newtheorem{lemma}[theorem]{Lemma}

\begin{document}
\title{Large-volume open sets in normed spaces without integral distances}
\author{Sascha Kurz and Valery Mishkin}

\address{Sascha Kurz, University of Bayreuth, D-95440 Bayreuth, Germany, sascha.kurz@uni-bayreuth.de\\
Valery Mishkin, York University, Toronto ON, M3J1P3 Canada, vmichkin@mathstat.yorku.ca}

\maketitle

\begin{abstract}
  We study open sets $\mathcal{P}$ in normed spaces $X$ attaining a large volume while avoiding pairs of points at integral
  distance.  The proposed task is to find sharp inequalities for the maximum possible $d$-dimensional volume. This problem can be
  viewed as an opposite to known problems on point sets with pairwise integral or rational distances. 
  %Most of our results
  %apply to the special case of Euclidean spaces ${\mathbb R}^d$ while several concepts can be transfered to arbitrary normed
  %spaces, which results in some new open questions for those spaces. 
\end{abstract}

\section{Introduction}
For quite some time it was not known whether there exist seven points in the Euclidean plane, no three on a line, no four on a circle, 
with pairwise integral distances. Kreisel and Kurz~\cite{1145.52010} found such a set of size~$7$, but it is unknown if there
exists one of size~$8$.\footnote{\highlight{We remark that, even earlier,  Noll and Simmons, see http://www.isthe.com/chongo/tech/math/n-cluster, 
found those sets in 2006 with the additional property of having integral coordinates, so-called $7_2$-clusters.}} The hunt for those \highlight{point} sets %in general position
was initiated by Ulam in 1945 by asking for a dense point set in the plane with pairwise rational distances.%, which is still widely open.

Here we study a kind of opposite problem, recently \highlight{considered} by the authors for Euclidean spaces, see \cite{mishkin_kurz}: Given 
a normed %(and \textbf{measureable})
space $X$, what is the maximum volume $f(X,n)$ of an open set $\mathcal{P}\subseteq X$ with $n$ connected
components without a pair of points at integral distance. We \highlight{drop} some technical assumptions for the normed spaces $X$
and mostly consider \highlight{the} Euclidean spaces $\mathbb{E}^d$ or $\mathbb{R}^d$ equipped with a $p$-norm.
\highlight{In Theorem~\ref{main_thm} we state an explicit formula for the Euclidean case $f(\mathbb{E}^d,n)$.}

\section{Basic notation and first observations}
We assume that our normed space $X$ admits a measure, which we denote by $\highlight{\lambda_X}$. By $\mathcal{B}_X$ we denote the \highlight{open}
ball \highlight{with diameter one} in $X$, i.e.\ the set of points \highlight{with} distance \highlight{smaller than} $\frac{1}{2}$ from a given center.
\highlight{In the special case} $X=(\mathbb{R}^d,\Vert\cdot\Vert_p)$ with 
$\Vert \left(x_1,\dots,x_d\right)\Vert_X=\Vert \left(x_1,\dots,x_d\right)\Vert_p:=\left(\sum_{i=1}^d |x_i|^p\right)^{\frac{1}{p}}$, where $p\in\mathbb{R}_{>0}\cup\highlight{\{}\infty\highlight{\}}$, we have
$$
  \lambda_X(\mathcal{B}_X)=\Gamma\left(\frac{1}{p}+1\right)^d   / \Gamma\left(\frac{d}{p}+1\right),
$$
\highlight{where $\Gamma$ denotes the famous Gamma function, i.e.\ the extension of the factorial function.}
In the Manhattan metric, i.e.\ $p=1$, the volumes of the resulting cross-polytopes \highlight{equal} $\frac{1}{d!}$ and in the maximum norm, i.e.\
$p=\infty$, the volumes of the resulting hypercubes \highlight{equal} $1$.

At first we observe that the diameter of a connected component $\mathcal{C}$ of a set $\mathcal{P}\subseteq X$ avoiding
integral distances is at most $1$.\footnote{\highlight{We remark that the maximum volume of a set with diameter $1$ in $X$ is
at most $\lambda_X(\mathcal{B}_X)$, see e.g.\ \cite{Melnikov}.}} Otherwise we \highlight{can} consider two points $u,v\in \mathcal{C}$
having a distance larger \highlight{than} $1$
and conclude the existence of a point $w\in \mathcal{C}$ on the curve connecting $u$ with $v$ in $\mathcal{C}$ such that the
distance between $u$ and $v$ is exactly $1$.

Given two points $u,v\in X$, where $\Vert v-u\Vert_X=1$, we may consider the line
$
  \mathcal{L}:=\{u+\alpha (v-u)\mid \alpha\in\mathbb{R}\}\highlight{\subseteq} X
$.
The restriction of $X$ to $\mathcal{L}$ yields another, one-di\-men\-sional, normed space, where we can \highlight{pose} the same question. 

We consider the map $\varphi: \mathcal{L}\to [0,1)$, $p\mapsto \alpha\mod 1$ where 
$p=u+\alpha(v-u)\in \mathcal{L}$. If $\varphi$ is not injective on $\mathcal{P}\cap\mathcal{L}$, 
the set $\mathcal{P}$ contains a pair of point an integral distance apart. Since the map
$\varphi$ is length preserving modulo $1$, that is, $\Vert p_1-p_2\Vert_X \mod 1=| \varphi(p_1)-\varphi(p_2)|$,
its restriction to the connected components of $\mathcal{P}\cap \mathcal{L}$ is length preserving. 
Thus, we conclude the necessary condition for an open set avoiding integral distances that the
length of each intersection with a line is at most $1$.

%We consider the map $\varphi:\mathcal{L}\rightarrow [0,1)$, $p\mapsto \min \bigl\{\Vert p-\highlight{k}(v-u)\Vert_X\mid k\in\mathbb{Z}\bigr\}$.
%If $\varphi$ is not injective on $\mathcal{P}\cap\mathcal{L}$, the set $\mathcal{P}$ contains a pair of points at integral distance. Since the
%map $\varphi$ is \highlight{length}-preserving we conclude the necessary condition for an open set avoiding integral distances that the
%\highlight{length} of each intersection with a line is at most $1$.

Having those two necessary conditions, i.e.\ \highlight{the diameter} of \highlight{each} connected component \highlight{is at most $1$} and
\highlight{the length of} each line intersection \highlight{is at most $1$}, 
at hand we define $l(X,n)$ as the maximum volume of open sets in $X$ with $n$ connected components, which satisfy the \highlight{two} necessary conditions. 
\highlight{We thus} have $f(X,n)\le l(X,n)\le n\cdot \lambda_X(\mathcal{B}_X)$ for all $n\in\mathbb{N}$.

In \cite{mishkin_kurz} the authors \highlight{provided an example} of \highlight{a connected} open set $U\subseteq\mathbb{R}^d$
such that the intersection \highlight{of $U$} with each line has a total length of at most $1$ but the volume of \highlight{$U$} is unbounded.

Based on a simple averaging argument, any given upper bound on one of the two introduced maximum volumes for $n$ components
induces an upper bound for $k\ge n$ components in the same normed space $X$:
\begin{lemma}For each $k\ge n$ we have,
  \label{lemma_averaging}
  \begin{itemize}
    \item $l(X,k)\le \frac{k}{n}\cdot \Lambda$ whenever $l(X,n)\le \Lambda$;
    \item $f(X,k)\le \frac{k}{n}\cdot \Lambda$ whenever $f(X,n)\le \Lambda$.
  \end{itemize}
\end{lemma}

For lower bounds we consider the following construction. Given a small constant $0<\varepsilon<\frac{1}{n}$, we arrange one open ball of diameter 
$1-(n-1)\varepsilon$ and $n-1$ open balls of diameter $\varepsilon$ each, so that there centers are aligned on a line and that they are non-intersecting.
Since everything fits into an open ball of diameter $1$ there cannot be a pair of points at integral distance. \highlight{As} $\varepsilon\to 0$ the
volume of the constructed set approaches $\lambda_X(\mathcal{B}_X)$, so that we have
$$
   \lambda_X(\mathcal{B}_X)\le f(X,n)\le l(X,n).
$$
\highlight{We then} have $f(X,1)\le l(X,1)=\lambda_X(\mathcal{B}_X)$. \highlight{The} map $\varphi$ \highlight{shows that} equality is also
attained for normed spaces $X$ of dimension $1$. So,  in the following we consider sets consisting of at least two components and normed spaces of
dimension at least two.

\section{Two components}
\label{sec_two_components}
For one component the extremal example was the open ball of diameter $1$. By choosing the line through the centers of two balls of diameter $1$ we obtain
a line intersection of total length \highlight{two}, so that this cannot happen in an integral distance avoiding set. The idea to circumvent this fact is
to truncate the open balls in direction of the line connecting the centers so that both components have a width of almost $\frac{1}{2}$, see 
Figure~\ref{fig_two_component_construction_euclidean_plane} for the Euclidean plane $\mathbb{E}^2$.

\begin{figure}[htp]
  \begin{center}
    \includegraphics{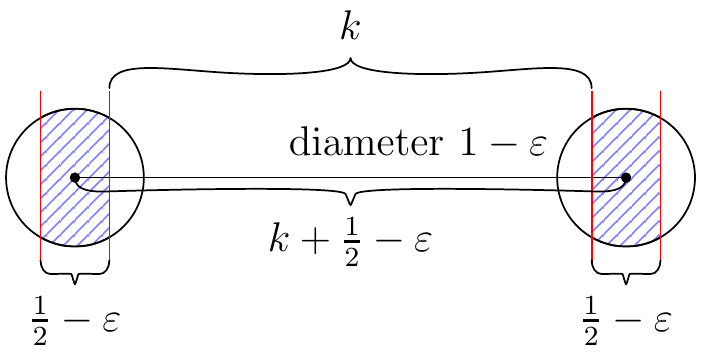}
    \caption{Truncated \highlight{disks} -- a construction of $2$ components in $\mathbb{E}^2$ without integral distances.}
    \label{fig_two_component_construction_euclidean_plane}
  \end{center}
\end{figure}

The volume $V$ of a convex body $\mathcal{K}\subset \mathbb{E}^d$ with diameter $D$ and minimal width $\omega$ is bounded above by
%In the $d$-dimensional Euclidean space $\mathbb{E}^d$  a convex body $\mathcal{K}$ with diameter $D$ and minimal width $\omega$ its $d$-dimensional
%volume $V$ is upper bounded by namely:
\begin{equation}
  \label{ie_isoperimetric_width}
  V\le \lambda_{\mathbb{E}^{d-1}}(\mathcal{B}_{\mathbb{E}^{d-1}})\cdot D^d\int_0^{\arcsin \frac{\omega}{D}}\cos ^d\theta\,\mbox{d}\theta,
\end{equation}
see e.g.\ \cite[Theorem 1]{1074.52004}. Equality holds \highlight{iff} $\mathcal{K}$ is the $d$-dimensional spherical symmetric slice with
diameter~$D$ and minimal width~$\omega$. We can easily check that the maximum volume of two $d$-dimensional spherical symmetric slices with
diameter $1$ each and minimal widths $\omega_1$ and $\omega_2$, respectively, so that $\omega_1+\omega_2\le 1$ is attained \highlight{at}
$\omega_1=\omega_2=\frac{1}{2}$, independently from the dimension.

Motivated by this fact we generally define $S_X$ \highlight{to be a} spherical symmetric slice with diameter $1$ and width $\frac{1}{2}$, i.e.\ a truncated
open ball. In the $d$-dimensional Euclidean case we have
\begin{equation}
  \lambda_{\mathbb{E}^d}(S_{\mathbb{E}^d})= \lambda_{\mathbb{E}^{d-1}}(\mathcal{B}_{\mathbb{E}^{d-1}})\int\limits_0^{\frac{\pi}{6}}
  \cos ^d\theta\,\mbox{d}\theta.
\end{equation}
The truncated \highlight{disc} in dimension $d=2$ has an area of $\frac{\sqrt{3}}{8}+\frac{\pi}{12}\approx 0.4783$.

\begin{figure}[htp]
  \begin{center}
    \includegraphics{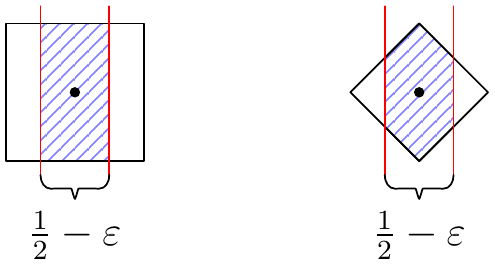}
    \caption{Spherical slices in dimension $2$ for the maximum norm, $p=\infty$, and the Manhattan metric, $p=1$.}
    \label{fig_spherical_slices}
  \end{center}
\end{figure}

On the left hand side of Figure~\ref{fig_spherical_slices} we have drawn the spherical slice, \highlight{i.e.\ a truncated ball,} in dimension $2$ for the maximum norm, i.e.\ $p=\infty$.
For general dimension $d$ we have $\lambda_{X}(S_X)=\frac{1}{2}$. On the right hand side of Figure~\ref{fig_spherical_slices} we have drawn the spherical
slice in dimension $2$ for the Manhattan metric, i.e.\ $p=1$. Here we have $\lambda_{X}(S_X)=\frac{1}{d!}-\frac{1}{2^{d-2}(d-1)!}$.

Since the line through the upper left \textit{corner} of the left component and the lower right \textit{corner} of the right component should have an
intersection with the \highlight{shaded region} of total length at most $1$, we consider truncated open balls of diameter $1-\varepsilon$ and width $\frac{1}{2}-\varepsilon$, 
see Figure~\ref{fig_two_component_construction_euclidean_plane}. For some special normed spaces $X$ we can choose $\varepsilon>0$ and move the centers of
the two components sufficiently  away from each other \highlight{such that} we can guarantee that no line intersection has a total length of more \highlight{than} $1$.

\begin{lemma}
   \label{lemma_two_components}
   For arbitrary dimension $d\ge 2$ and normed spaces $X=(\mathbb{R}^d,\Vert\cdot\Vert_p)$ with \highlight{$1<p<\infty$} we have $l(X,2)\ge f(X,2)\ge 2\lambda_{X}(S_{X})$.
\end{lemma}
\begin{proof}
  For a given small $\varepsilon>0$ place a truncated open ball of diameter $1-\varepsilon$ and width $\frac{1}{2}-\varepsilon$ with its center at
  the origin and a second copy so that the two centers are at distance $k+\frac{1}{2}-\varepsilon$, see
  \highlight{Figure}~\ref{fig_two_component_construction_euclidean_plane}
  for $p=d=2$. Since both components have \highlight{diameters smaller} than $1$ there cannot be a pair of points at integral distance within the same
  component. So let
  $a=(a_1,\dots,a_d)$ be a point of the left and $b=(b_1,\dots,b_d)$ a point of the right component, where we assume that the centers of the components
  are moved apart along the first coordinate axis. By construction the distance between $a$ and $b$ is at least $k$. Since $\left|a_i-b_i\right|<1-\varepsilon$
  for all $2\le i\le d$ and $\left|a_1-b_1\right|<k+1-2\varepsilon$ the distance between $a$ and $b$ is less than
  $$
    \Bigl((1-\varepsilon)^p\cdot(d-1)+(k+1-2\varepsilon)^p\Bigr)^{\frac{1}{p}}.
  $$
  By choosing a sufficiently large integer $k$ we can guarantee that this term is at most $k+1$, so that there is no pair of points at integral distance.
  Finally, we consider the limit \highlight{as} $\varepsilon\to 0$.
\end{proof}

We conjecture that the lower bound from Lemma~\ref{lemma_two_components} is sharp for all $p>1$ and remark that this is true for the Euclidean case $p=2$ 
\highlight{by} Inequality~(\ref{ie_isoperimetric_width}). %For the technical details we refer the interested reader to \cite{mishkin_kurz}.

\section{Relation to finite point sets with pairwise odd integral distances}
In this section we restrict the connected components to open balls of diameter $\frac{1}{2}$. We remark that if an integral distance avoiding set
contains an open ball of diameter $0<D<1$, which fits into one of its components, then the other components can contain open balls of diameter at most $1-D$. 
One can easily check that the maximum volume of the entire set $\mathcal{P}$ is attained \highlight{at} $D=\frac{1}{2}$, at least for $p$-norms and
dimensions $d\ge 2$. If
the set $\mathcal{P}$, i.e.\ the collection of $n$ open balls of diameter $\frac{1}{2}$, does not contain a pair of points at integral distance, then
the mutual distances between centers of different balls have to be elements of $\mathbb{Z}+\frac{1}{2}$. \highlight{Therefore dilating} $\mathcal{P}$ by a factor of
$2$ \highlight{yields the set} $\mathcal{Q}$ of the centers of the \highlight{$n$} balls with pairwise odd integral distances. 

\highlight{However,} for Euclidean spaces $\mathbb{E}^d$, \highlight{it is known, see \cite{0277.10021},} that $|\mathcal{Q}|\le d+2$, where equality is possible
if and only if $d\equiv 2\pmod {16}$. It would be interesting to \highlight{determine} the maximum number of odd integral distances in other normed spaces.

\section{Large-volume open sets with diameter and maximum length of line intersections at most one}
Assuming that the construction using truncated open balls from Section~\ref{sec_two_components} is best possible or, at the very least, competitive, we can try
to arrange $n$ copies of those $S_X$. Since we have to \highlight{control} that each line meets at most two components we cannot arrange the centers on 
\highlight{a certain} line. On the other hand, the cutting \highlight{directions}, i.e.\ the \highlight{directions} where we cut of the caps from the open
balls, should be almost equal. To \highlight{meet} both \highlight{requirements}, we
arrange the centers of the components on a parabola, where each component has diameter $1-\varepsilon$ and width $\frac{1}{2}-\varepsilon$, for a small
constant $\varepsilon>0$, see Figure~\ref{fig_n_component_construction} for an example in \highlight{$\mathbb{E}^2$}.

\begin{figure}[htp]
  \begin{center}
    \includegraphics{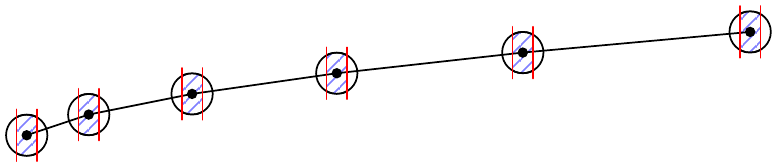}
    \caption{Truncated \highlight{discs} -- arranged on a parabola.}
    \label{fig_n_component_construction}
  \end{center}
\end{figure}

\begin{lemma}
   \label{lemma_n_components}
   For arbitrary dimension $d\ge 2$, $n\ge 2$, and normed spaces $X=(\mathbb{R}^d,\Vert\cdot\Vert_p)$ with $p>1$ we have $l(X,n)\ge n\lambda_{X}(S_{X})$.
\end{lemma}
\begin{proof}
  For a given small $\varepsilon>0$ consider $n$ truncated $d$-dimensional balls $S_X$ of width $\frac{1}{2}-\varepsilon$, where the truncation is oriented
  in the direction of the $y$-axis, with centers located at $\left(i\cdot k,i\cdot k^2,0,\dots,0\right)$ and diameter $1-\varepsilon$ for
  $1\le i\le n$, see Figure~\ref{fig_n_component_construction}. \highlight{For $k$ large,} there is no line intersecting three or
  more components. It remains to \highlight{check} that each line meeting two $d$-dimensional balls centered at $C_1$, $C_2$ has
  an intersection of \highlight{length} at most $\frac{1}{2}$ with each of the \highlight{truncated} balls. This can be done by performing a similar
  calculation as in the proof of Lemma~\ref{lemma_two_components}. Again, we consider the \highlight{limiting configuration as} $\varepsilon\to 0$.
\end{proof}

We conjecture that the lower bound \highlight{in} Lemma~\ref{lemma_n_components} is sharp.

\section{Using results from Diophantine Approximation}
In order to modify the construction from the previous section to obtain a lower bound for $f(X,n)$, we use results from Diophantine Approximation 
\highlight{and:}
%. A possibly well known result on the lengths of the diagonals of a regular $p$-gon in the Euclidean plane $\mathbb{E}^2$ is the following:
\begin{lemma} 
  \label{lemma_diagonals_p_gon}
  Given an odd prime $p$, let $\alpha_j=\frac{\zeta_j-\zeta_{2p-j}}{i}$ for $1\le j\le \frac{p-1}{2}$, where the $\zeta_j$ are $2p$th
  roots of unity. Then the $\alpha_j$ are irrational and linearly independent over $\mathbb{Q}$.
\end{lemma}
A proof, based on a theorem of vanishing sums of roots of unity by Mann \cite{0138.03102}, is given in \cite{mishkin_kurz}.

This result can be \highlight{applied to construct} sets avoiding integral distances as follows. \highlight{We} fix an odd prime $p$ with $p\ge n$.  For
each integer $k\ge 2$ and each $\frac{1}{4}>\varepsilon>0$ we consider a regular $p$-gon $P$ with side lengths $2k\cdot\sin\Bigl(\frac{\pi}{p}\Bigr)$,
i.e.\ with \highlight{circumradius} $k$. At $n$ arbitrarily chosen vertices of the $p$-gon $P$ we place the centers of $d$-dimensional open balls with diameter
$1-\varepsilon$. Since \highlight{the diameter of each of the $n$ components is} less than $1$ there is no pair of points at integral distance inside
one of these $n$
components. For each pair of centers \highlight{$c_1$} and \highlight{$c_2$}, we cut off the corresponding two components such that each component has
a width of $\frac{1}{2}-\varepsilon$ in that direction, see Figure~\ref{fig_pentagon_construction} for an example with $n=p=5$.

\begin{figure}[htp]
  \begin{center}
    \includegraphics[width=5cm]{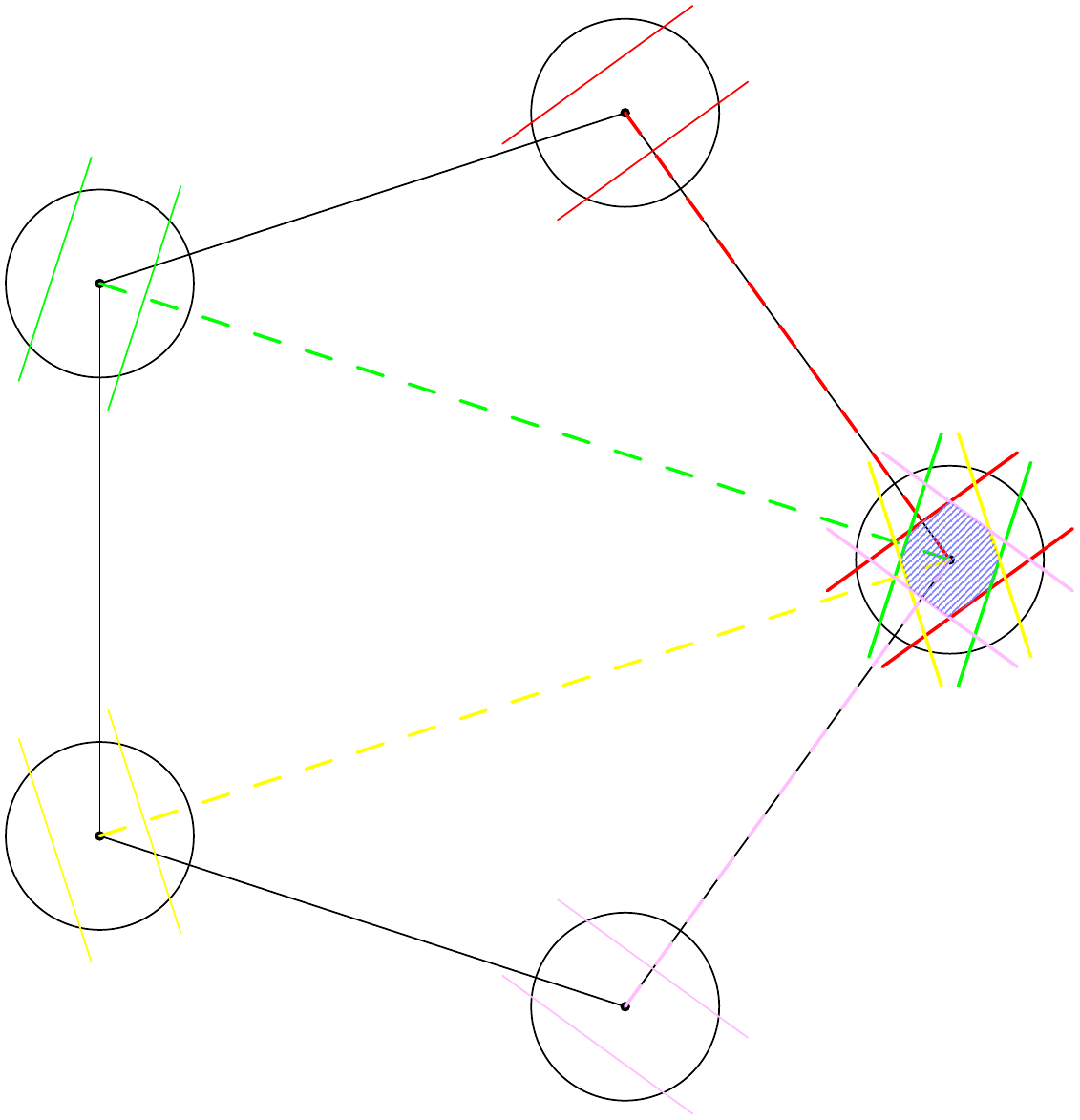}
    \caption{$p$-gon construction: Integral distance avoiding point set for $d=2$  and $p=n=5$.}
    \label{fig_pentagon_construction}
  \end{center}
\end{figure}

Next we consider two points $a$ and $b$ from \highlight{different} components. \highlight{We denote by $\alpha$} the distance of the centers of the
corresponding components. From the triangle inequality we conclude
$$
  \alpha-\left(\frac{1-2\varepsilon}{2}\right)<\operatorname{dist}(a,b)<\alpha+\left(\frac{1-2\varepsilon}{2}\right).
$$
Since the occurring distances $\alpha$ are given by $2k\sin\Bigl(\frac{j\pi}{p}\Bigr)$ for $1\le j\le \frac{p-1}{2}$ we look
for a \highlight{solution of the following system of inequalities}
\begin{equation}
  \left\{2k\cdot\sin\!\left(\frac{j\pi}{p}\right)-\frac{1}{2}+\varepsilon\right\}\le 2\varepsilon
\end{equation}
with $k\in\mathbb{N}$, where $\highlight{\{}\beta\highlight{\}}$ denotes the positive fractional part of a real number~$\beta$, i.e.\ there exists an integer $l$ with
$\beta=l+\highlight{\{}\beta\highlight{\}}$ and $0\le \highlight{\{}\beta\highlight{\}}<1$. By Lemma~\ref{lemma_diagonals_p_gon} the 
factors $2\sin\!\left(\frac{j\pi}{p}\right)$ are
irrational and linearly independent over $\mathbb{Q}$, so by Weyl's Theorem \cite{46.0278.06} the systems admit a solution for all $k$. (Actually
we only use the denseness result, which Weyl \highlight{himself attributed} to Kronecker.) We call the just described construction the $p$-gon construction.

\highlight{These ingredients we provided in more detail in \cite{mishkin_kurz} enable us to establish} the conjectured exact \highlight{values} of
\highlight{the function} $f\Bigl((\mathbb{E}^d,\Vert\cdot\Vert_2),n\Bigr))$:

\begin{theorem}
  \label{main_thm}
  For all $n,d\ge 2$ we have $$f\Bigl(\mathbb{E}^d,n\Bigr)=n\cdot\lambda_{\mathbb{E}^d}(S_{\mathbb{E}^d}).$$
\end{theorem}
\begin{proof}(Sketch)\\
For a given number $n$ of components we choose an increasing sequence of primes $n<p_1<p_2<\dots$. For each $i=1,2,\dots$ we 
consider the $p$-gon construction with $p=p_i$ and $n$ neighbored vertices. If the scaling factor $k$ tends to infinity the 
$d$-dimensional volume of the resulting sequence of integral distance avoiding sets tends to the volume of the following set:
Let $\mathcal{P}_i$ arose as follows. Place an open ball of diameter $1$ at $n$ neighbored vertices of a regular $p_i$-gon with
radius $2n^2$. Cut off caps in the direction of the lines connecting each pair of centers so that the components have a width of
$\frac{1}{2}$ in that direction. In order to estimate the volume of $\mathcal{P}_i$ we consider another set $\mathcal{T}_i$, which 
arises as follows. Place an open ball of diameter $1-\Delta_i$ at $n$ neighbored vertices of a regular $p_i$-gon with
radius $2n^2$. Cut off caps in the direction of the $x$-axis so that the components have a width of $\frac{1}{2}-\Delta_i$. One can
suitably choose $\Delta_i$ so that $\mathcal{T}_i$ is \highlight{contained} in $\mathcal{P}_i$. Since the centers of the components 
of $\mathcal{P}_i$ tend to be \highlight{aligned, as $i$ increases,} $\Delta_i$ tends to $0$ as $i$ approaches infinity.
\highlight{We thus} have $\lim_{i\to\infty} \lambda_{\mathbb{E}^d}(\mathcal{T}_i)=n\cdot\lambda_{\mathbb{E}^d}(S_{\mathbb{E}^d})$.
\end{proof}

\section{Conclusion}
We have proposed the question for the maximum volume of an open set $\mathcal{P}$ consisting of $n$ components in an arbitrary normed space $X$
avoiding integral distances. For the Euclidean plane those sets need to have upper density $0$, see \cite{0738.28013}. 
%A similar result for finite field geometries can be found in \cite{Covert2011423}.
Theorem~\ref{main_thm} proves a conjecture \highlight{stated in} \cite{mishkin_kurz} and some of the concepts \highlight{have been transferred} to more
general spaces. Nevertheless, many problems remain unsolved and provoke further research.

\small

%% \bibliography{IntegralDistanceAvoiding}

\begin{thebibliography}{1}

\bibitem{0738.28013}
H.~Furstenberg, Y.~Katznelson, and B.~Weiss, \emph{Ergodic theory and
  configurations in sets of positive density}, Mathematics of Ramsey theory,
  Coll. Pap. Symp. Graph Theory, Prague/Czech., Algorithms Comb. 5, 184--198,
  1990.

\bibitem{0277.10021}
R.L. Graham, B.L. Rothschild, and E.G. Straus, \emph{Are there n + 2 points in
  {E}$^n$ with odd integral distances?}, Amer. Math. Mon. \textbf{81} (1974),
  21--25.

\bibitem{1074.52004}
M.A. Hern\'andez~Cifre, G.~Salinas, and S.~Segura~Gomis, \emph{Two optimization
  problems for convex bodies in the $n$-dimensional space}, Beitr\"age Algebra
  Geom. (Contributions to Algebra and Geometry) \textbf{45} (2004), no.~2,
  549--555.

\bibitem{1145.52010}
T.~Kreisel and S.~Kurz, \emph{There are integral heptagons, no three points on
  a line, no four on a circle}, Discrete Comput. Geom. \textbf{39} (2008),
  no.~4, 786--790.

\bibitem{mishkin_kurz}
Sascha Kurz and Valery Mishkin, \emph{Open sets avoiding integral distances},
  Discrete \& Computational Geometry \textbf{50} (2013), no.~2, 99--123.

\bibitem{0138.03102}
H.B. Mann, \emph{On linear relations between roots of unity}, Mathematika,
  Lond. \textbf{12} (1965), 107--117.

\bibitem{Melnikov}
Mark~Samuilovich Mel'nikov, \emph{Dependence of volume and diameter of sets in
  an n-dimensional {B}anach space}, Uspekhi Matematicheskikh Nauk \textbf{18}
  (1963), no.~4, 165--170.

\bibitem{46.0278.06}
H.~Weyl, \emph{{\"U}ber die {G}leichverteilung von {Z}ahlen mod.~{E}ins}, Math.
  Ann. \textbf{77} (1916), 313--352 (German).

\end{thebibliography}
%% \bibdata{IntegralDistanceAvoiding}
%% \bibliographystyle{amsplain}

\end{document}